\documentclass[11pt]{article}
\usepackage[a4paper]{anysize}\marginsize{1.5cm}{1.5cm}{1.5cm}{1cm}
\pdfpagewidth=\paperwidth \pdfpageheight=\paperheight
\usepackage{amsfonts,amssymb,amsthm,amsmath,eucal,tabu,url}
\usepackage{pgf}
 \usepackage{array}
 \usepackage{pstricks}
 \usepackage{pstricks-add}
 \usepackage{pgf,tikz}
 \usetikzlibrary{automata}
 \usetikzlibrary{arrows}
 \usepackage{indentfirst}
\theoremstyle{plain}
\newtheorem{thm}{Theorem}[section]
\newtheorem{theorem}[thm]{Theorem}

\newtheorem{proposition}[thm]{Proposition}

\theoremstyle{definition}

\newtheorem{remark}[thm]{Remark}

\newtheorem{thevarthm}[thm]{\varthmname}

\newenvironment{varthm*}[1]{\trivlist\item[]{\bf #1.}\it}{\endtrivlist}

\renewcommand\geq{\geqslant}

\renewcommand\leq{\leqslant}

\newcommand\be{\begin{eqnarray*}}
\newcommand\ee{\end{eqnarray*}}

\newcommand\newop[2]{\def#1{\mathop{\rm #2}\nolimits}}
\newop\log{log}
\newop\ord{ord}
\newop\Gal{Gal}
\newop\SL{SL}
\newop\Bl{Bl}
\newop\mult{mult}
\newop\mass{mass}
\newop\div{div}
\newop\codim{codim}
\newop\sing{sing}
\newop\vdim{vdim}
\newop\edim{edim}
\newop\Ass{Ass}
\newop\size{size}
\newop\reg{reg}
\newop\satdeg{satdeg}
\newop\supp{supp}
\newop\Neg{Neg}
\newop\Nef{Nef}
\newop\Nefh{Nef_H}
\newop\Eff{Eff}
\newop\Zar{Zar}
\newop\MB{MB}
\newop\MBxC{MB\mathit{(x,C)}}
\newop\NnB{NnB}
\newop\Bigg{Big}
\newop\Effbar{\overline{\Eff}}

\def\keywordname{{\bfseries Keywords}}%
\def\keywords#1{\par\addvspace\medskipamount{\rightskip=0pt plus1cm
\def\and{\ifhmode\unskip\nobreak\fi\ $\cdot$
}\noindent\keywordname\enspace\ignorespaces#1\par}}
\def\subclassname{{\bfseries Mathematics Subject Classification
(2000)}\enspace}
\def\subclass#1{\par\addvspace\medskipamount{\rightskip=0pt plus1cm
\def\and{\ifhmode\unskip\nobreak\fi\ $\cdot$
}\noindent\subclassname\ignorespaces#1\par}}

\begin{document}
\title{The orbifold Langer-Miyaoka-Yau inequality and Hirzebruch-type inequalities}
\author{Piotr Pokora}

\date{\today}
\maketitle
\par\vspace*{.001\textheight}{\centering \emph{To Professor Kamil Rusek, on the occasion of his 70th birthday.}\par}
\begin{abstract}
Using Langer's variation on the Bogomolov-Miyaoka-Yau inequality \cite[Theorem 0.1]{Langer} we provide some Hirzebruch-type inequalities for curve arrangements in the complex projective plane.
\subclass{14C20, 32S22, 52C35}
\end{abstract}

%*****************************************************************************
\section{Introduction}
In 2003 A. Langer \cite{Langer} has shown the following beautiful variation on the classical Bogomolov-Miyaoka-Yau inequality for a normal surface $X$ with a boundary divisor $D$.

\begin{theorem}
Let $(X,D)$ be a normal projective surface with a $\mathbb{Q}$-divisor $D = \sum_{i}a_{i}D_{i}$ with $0 \leq a_{i} \leq 1$. Assume that the pair $(X,D)$ is log canonical and $K_{X}+D$ is $\mathbb{Q}$-effective. Then
\begin{equation}
\label{LMY}
(K_{X}+D)^{2} \leq 3e_{orb}(X,D).
\end{equation}
Moreover, if equality holds, then $K_{X}+D$ is nef.
\end{theorem}
In the above formulae, $e_{orb}(X,D)$ denotes the global orbifold number for $(X, \sum_{i}a_{i}D_{i})$, i.e., 
$$e_{orb}(X,D) = e_{top}(X) - \sum_{i}a_{i}e_{top}(D_{i}\setminus \text{Sing}(X,D)) + \sum_{x \in \text{Sing}(X,D)} (e_{orb}(x,X,D)-1),$$
and by $e_{orb}(x,X,D)$ we denote the local orbifold Euler number at $x$ \cite[Definition 3.1]{Langer}. For us the most important property is that local orbifold Euler numbers are analytic in their nature. In the present note we would like to obtain some Hirzebruch-type inequalities for curve arrangements in the complex projective plane such that all irreducible components are smooth and have pairwise transversal intersection points, i.e., all singularities are \emph{ordinary} and locally look like $\{x^{k} = y^{k}\}$. For those singularities, Langer \cite[Theorem 8.7]{Langer} computed their local orbifold Euler numbers using lines in $\mathbb{C}^{2}$.
\begin{proposition}
\label{mult}
Let $L_{1}, ..., L_{n}$ be $n$ distinct lines in $\mathbb{C}^{2}$ passing through $0$. Set $D = \sum_{i=1}^{n}a_{i}L_{i}$, where $0 \leq a_{1} \leq ... \leq a_{n} \leq 1$, and $a = \sum_{i=1}^{n}a_{i}$. Then
\begin{displaymath}
e_{orb}(0;\mathbb{C}^{2}, D) = \left\{ \begin{array}{ll}
0 & \text{if } a > 2 \\
(1-a+a_{n})(1-a_{n}) & \text{if } 2a_{n} \geq a
\end{array} \right.
\end{displaymath}
and
$$e_{orb}(0;\mathbb{C}^{2}, D) \leq \bigg(1 - \frac{a}{2}\bigg)^{2}$$
if $2a_{n} \leq a \leq 2$.
\end{proposition}
Now we would like to look at the case of curves in the complex projective plane. If $C$ is a reduced curve of degree $d$, then by \cite[p.~820]{BK} we have
$$e_{top}(C) = -d(d-3) + \sum_{p \in \text{Sing}(C)} \mu_{p},$$
where $\mu_{p}$ denotes the Milnor number of a singular point $p \in \text{Sing}(C)$.

It is easy to observe that if $(\mathbb{P}^{2}_{\mathbb{C}},\alpha C)$ is a log canonical pair for a suitable 
$\alpha$, then we can write (\ref{LMY}) as follows \cite[Section 11]{Langer}:
\begin{equation}
\label{Langer1}
\sum_{p \in \text{Sing}(C)} 3 \bigg( \alpha(\mu_{p}-1) + 1 - e_{orb}(p, \mathbb{P}^{2}_{\mathbb{C}}, \alpha C) \bigg) \leq (3\alpha -\alpha^{2})d^{2} - 3\alpha d.
\end{equation}

The main idea of this note is to provide some Hirzebruch-type inequalities for arrangements of curves in the complex projective plane. We will show that using Langer's variation on the Bogomolov-Miyaoka-Yau inequality \cite{Miyaoka} one can obtain rather elementary proofs of them, i.e., we do not need to pass to Hirzebruch's construction which involves abelian covers branched along arrangements of curves. It is worth pointing out that applied methods allow to deal with configurations of curves having different degrees of irreducible components in a much easier way than in Hirzebruch's construction, especially one can avoid very complicated conditions under which the resulting surfaces has non-negative Kodaira dimension, a crucial condition to apply the classical Bogomolov-Miyaoka-Yau inequality. Another motivation is to show that using Langer's approach one can obtain, somehow surprisingly, the so-called `quadratic right-hand side in Hirzebruch's inequality' for a large class of curve arrangements.
\section{Hirzebruch-type inequalities}
In this section, we assume that all curve arrangements $\mathcal{C} \subset \mathbb{P}^{2}_{\mathbb{C}}$ have ordinary singularities and every irreducible component of $\mathcal{C}$ is smooth. For a given arrangement of curves $\mathcal{C}$ we denote by $t_{r}$ the number of $r$-fold points, i.e., points where $r$-curves from the arrangement meet. Moreover, we define for $i \in \{0,1,2\}$ the following numbers
$$f_{i} = \sum_{r\geq 2} r^{i}t_{r},$$
and finally we will use the following elementary observations: $$ \sum_{p \in \text{Sing}(\mathcal{C})} m_{p} = f_{1}, \quad \sum_{p \in \text{Sing}(\mathcal{C})} m_{p}^{2} = f_{2},$$ where by $m_{p}$ we denote the multiplicity of $p \in \text{Sing}(C)$ -- in our situation this is equal to the number of curves passing through $p$.

Now we present the main results. The first one is devoted to line-conic arrangements in the complex projective plane and this result, according to the author's knowledge, is the first result providing some constraints on the combinatorics of such arrangements. Moreover, these arrangements seem to be interesting in the context of a generalized Terao's conjecture.  As we can see in \cite[Example 4.2]{Hal}, it is possible to find two configurations of lines and conics with ordinary singularities, which are combinatorially identical, but only one of them is free.
\begin{theorem}
Let $\mathcal{LC} = \{L_{1}, ..., L_{l}, C_{1}, ..., C_{k}\}$ be an arrangement of $l$ lines and $k$ conics such that $t_{r} = 0$ for $r > \frac{2(l+2k)}{3}$. Then one has
$$t_{2} + \frac{3}{4}t_{3} + (4k+2l-4)k \geq l + \sum_{r\geq 5}\bigg(\frac{r^{2}}{4}-r\bigg)t_{r}.$$
\end{theorem}
\begin{proof}
For an arrangement $\mathcal{LC}$ let us denote by $C = L_{1} + ... + L_{l} + C_{1} + ... + C_{k}$ the associated divisor. First of all, we need to choose $\alpha$ in such a way that $K_{\mathbb{P}^{2}} + \alpha C$ is effective and log-canonical. Thus 
$$\frac{3}{\ell + 2k} \leq \alpha \leq \frac{2}{r_{max}},$$ where $r_{max}$ denotes the maximal possible multiplicity of singular points in $\mathcal{LC}$. This implies in particular that $r_{max} \leq \frac{2\ell + 4k}{3}$. Let us now choose $\alpha \in [ \frac{3}{\ell + 2k}, \frac{2}{r_{max}}]$. Our aim is to apply (\ref{Langer1}) in the above setting. We start with the left-hand side. Using the fact that for a singular point $p$ we have $\mu_{p} = (m_{p}-1)^{2}$ (see for instance \cite[Remark 2.8]{Ploski}), one has

$$L: \, \sum_{p \in \text{Sing}(\mathcal{LC})} 3 \bigg( \alpha(\mu_{p} - 1) + 1 - e_{orb}(p,\mathbb{P}^{2}, \alpha C) \bigg) = \sum_{p \in \text{Sing}(\mathcal{LC})} 3\bigg(\alpha (m_{p}^{2} - 2m_{p}) + 1 - e_{orb}(p,\mathbb{P}^{2}, \alpha C) \bigg) = $$
$$3\alpha (f_{2} - 2f_{1}) + 3f_{0} -3 \sum_{p \in \text{Sing}(\mathcal{LC})} e_{orb}(p,\mathbb{P}^{2}, \alpha C).$$
We need to use Proposition \ref{mult}. Since all $a_{1} = ... = a_{l+k} = \alpha$, if $p$ is a double point, then 
$$e_{orb}(p, \mathbb{P}^{2}, \alpha C) = (1 - \alpha)^{2},$$
and for points $p$ with multiplicities $3 \leq m_{p} = r \leq r_{max}$ one has
$$e_{orb}(p, \mathbb{P}^{2}, \alpha C) \leq \bigg(1 - \frac{\alpha r}{2} \bigg)^{2}.$$
This leads to
$$3\alpha (f_{2} - 2f_{1}) + 3f_{0} - 3\sum_{r\geq 2}t_{r} \bigg(1 - \frac{\alpha r}{2}\bigg)^{2} \leq 3\alpha (f_{2} - 2f_{1}) + 3f_{0} -3 \sum_{p \in \text{Sing}(\mathcal{LC})} e_{orb}(p,\mathbb{P}^{2}, \alpha C),$$
and finally we obtain
$$ 3\alpha f_{2} - 3 \alpha f_{1} - \frac{3}{4}\alpha^{2} f_{2} \leq 3\alpha (f_{2} - 2f_{1}) + 3f_{0} -3 \sum_{p \in \text{Sing}(\mathcal{LC})} e_{orb}(p,\mathbb{P}^{2}, \alpha C).$$
Let us come back to the right-hand side. First of all, observe that
$${ l \choose 2} + 4 { k \choose 2 } + 2kl = \sum_{r\geq 2}t_{r} {r \choose 2}$$
assuming also that ${0 \choose 2} = {1 \choose 2} = 0$. With $d = 2k + l$ we can rewrite the above combinatorial equality as
$$d^{2} = f_{2} - f_{1} + 4k + l.$$
Using this identity, we get
$$R: (3\alpha - \alpha^{2}) d^{2} - 3\alpha d = (3\alpha - \alpha^{2})(f_{2} - f_{1} +4k + l) - 3\alpha (2k+l) = 3\alpha f_{2} -3\alpha f_{1} -\alpha^{2} f_{2} - \alpha^{2} f_{1} + 6\alpha k - \alpha^{2}(4k+l).$$
Since $L\leq R$, we obtain
$$ 3\alpha f_{2} - 3 \alpha f_{1} - \frac{3}{4}\alpha^{2} f_{2}  \leq 3\alpha f_{2} -3\alpha f_{1} -\alpha^{2} f_{2} - \alpha^{2} f_{1} + 6\alpha k - \alpha^{2}(4k+l),$$
and this provides
$$ \bigg( \frac{6}{\alpha} - 4\bigg)k + t_{2} + \frac{3}{4}t_{3} \geq l + \sum_{r\geq 5}\bigg(\frac{r^{2}}{4} - r \bigg)t_{r}.$$
In particular, taking $\alpha^{-1} = (2k+l)/3$ one has
$$t_{2} + \frac{3}{4}t_{3} + (4k+2l-4)k \geq l + \sum_{r\geq 5}\bigg(\frac{r^{2}}{4}-r\bigg)t_{r},$$
which completes the proof.
\end{proof}
Now we consider the case when all components have the same degree. 
\begin{theorem}
\label{degree:d}
Let $\mathcal{C} = \{C_{1}, ..., C_{k}\} \subset \mathbb{P}^{2}_{\mathbb{C}}$ be an arrangement of $k$ curves such that all irreducible components have degree $d \geq 1$. Moreover, we assume that $t_{r} = 0$ for $r > \frac{2dk}{3}$. Then one has
$$t_{2} + \frac{3}{4}t_{3} + d^{2}k(dk-k-1) \geq \sum_{r\geq 5}\bigg(\frac{r^{2}}{4} - r\bigg)t_{r}.$$
\end{theorem}
\begin{proof}
For a given arrangement $\mathcal{C}$ let us denotes by $C = C_{1} + ... + C_{k}$ the associated divisor. First of all, we need to choose $\alpha$ in such a way that $K_{\mathbb{P}^{2}} + \alpha C$ is effective and log-canonical. Thus 
$$\frac{3}{dk} \leq \alpha \leq \frac{2}{r_{max}},$$ where $r_{max}$ denotes the maximal multiplicity of singular points in $\mathcal{C}$. This implies in particular that $r_{max} \leq \frac{2dk}{3}$. Observe that if $d\geq 2$, then there are no constraints on multiplicities of singular points. Let us now choose $\alpha \in [ \frac{3}{dk}, \frac{2}{r_{max}}]$.
The proof is analogical as in the previous case. Since the local orbifold Euler number is analytic in nature, thus the left-hand side has the following form:
$$L: \, 3\alpha f_{2} - 3 \alpha f_{1} - \frac{3}{4}\alpha^{2} f_{2}.$$
Now we focus on the right-hand side. Recall that we have the following combinatorial equality
$$d^{2} {k \choose 2} = \sum_{r\geq 2} t_{r} { r \choose 2}.$$
This leads to 
$$d^{2}k^{2} = f_{2}-f_{1} +d^{2}k$$
and
$$R: \, (3\alpha -\alpha^{2})(d^{2}k^{2}) - 3\alpha dk = (3\alpha - \alpha^{2})(f_{2}-f_{1}+d^{2}k) - 3\alpha dk = $$
$$3\alpha f_{2} - 3\alpha f_{1} -\alpha^{2} f_{2} + \alpha^{2}f_{1} - \alpha^{2}d^{2}k + 3\alpha d^{2}k - 3\alpha dk.$$ 
Using $L\leq R$ we get
$$3\alpha f_{2} - 3 \alpha f_{1} - \frac{3}{4}\alpha^{2} f_{2} \leq 3\alpha f_{2} - 3\alpha f_{1} -\alpha^{2} f_{2} + \alpha^{2}f_{1} - \alpha^{2}d^{2}k + 3\alpha dk(d - 1).$$
Finally 
$$t_{2} + \frac{3}{4}t_{3} + \frac{3}{\alpha}dk(d-1) \geq d^{2}k + \sum_{r\geq 5} \bigg( \frac{r^{2}}{4}-r\bigg)t_{r}.$$
In particular, taking $\alpha^{-1} = \frac{dk}{3}$ one has
$$t_{2} + \frac{3}{4}t_{3} + d^{2}k(dk-k-1) \geq \sum_{r\geq 5} \bigg( \frac{r^{2}}{4}-r\bigg)t_{r}.$$
\end{proof}
\begin{remark}
Let us emphasize that one can also apply Hirzebruch's construction to curve arrangements in the complex projective plane such that each irreducible component has degree $d\geq 2$. As it was shown in \cite[p.~8]{PRSz} for $d\geq 3$ and in \cite{LT} for $d=2$, if $\mathcal{C} = \{C_{1}, ..., C_{k}\} \subset \mathbb{P}^{2}_{\mathbb{C}}$ is a such configuration with $t_{k}=0$, then
$$\bigg( \frac{7}{2}d^{2}-\frac{9}{2}d \bigg)k + t_{2} + t_{3} \geq \sum_{r\geq 5}(r-4)t_{r}.$$ 
\end{remark}
\begin{remark}
The case $d=1$ was announced in \cite[Lemma 2.2]{Bojanowski} as a consequence of \cite[Proposition 11.3.1]{Langer}. On the other hand, this particular inequality is a special case of a much stronger result from the same thesis \cite[Theorem 2.3]{Bojanowski}.
\end{remark}

\begin{remark}
Now we focus on the case $d=1$. Let $\mathcal{L}$ be a line arrangement in the complex projective space such that $t_{r} = 0$ for $r > \frac{2}{3}k$, then we obtain
\begin{equation}
\label{LineLanger}
t_{2} + \frac{3}{4}t_{3} \geq k + \sum_{r\geq 5} \bigg( \frac{r^2}{4}-r \bigg) t_{r}.
\end{equation}
It is worth pointing out that this is the strongest known inequality for line arrangements with $t_{r} = 0$ for $r > \frac{2}{3}k$. Recall that for line arrangements with $k \geq 6$ lines and $t_{k} = t_{k-1} = t_{k-2} = 0$ Hirzebruch proved the following inequalities (see \cite{Hirzebruch} for the first and \cite{Hirzebruch1} for the second inequality):
\begin{itemize}
\item $t_{2} + t_{3} \geq k + \sum_{r\geq 5} (r-4)t_{r}$,
\item $t_{2} + \frac{3}{4}t_{3} \geq k + \sum_{r \geq 5} (2r-9)t_{r}$,
\end{itemize}
and finally we have the following sequence of inequalities:
$$t_{2} + t_{3} \geq t_{2} + \frac{3}{4}t_{3} \geq k + \sum_{r\geq 5} \bigg( \frac{r^{2}}{4}-r\bigg)t_{r} \geq k + \sum_{r\geq 5}(2r-9)t_{r} \geq k + \sum_{r\geq 5}(r-4)t_{r}.$$
\end{remark}

In order to emphasize the fact that (\ref{LineLanger}) is the strongest known inequality in the mentioned class of line arrangements, we recall \cite[Example 11.3.2]{Langer}. Below we present a list of particular examples \cite[p.~210]{BHH87} for which the equality in (\ref{LineLanger}) holds:
\begin{enumerate}
\item \emph{Icosahedron arrangement} consisting of $15$ lines and $t_{2} = 15, t_{3} = 10, t_{5}=6$,
\item \emph{Klein's arrangement} consisting of $21$ lines and $t_{3}=28, t_{4}=21$,
\item \emph{Fermat's arrangements} consisting of $3n$ lines with $n\geq 3$, and $t_{3} = n^{2}, t_{n}=3$,
\item \emph{Hesse arrangement} consisting of $12$ lines and $t_{4}=9, t_{2} = 12$,
\item \emph{extended Hesse arrangement} consisting of $21$ lines and $t_{2}=36, t_{4} = 9, t_{5}=12$,
\item \emph{Wiman's arrangement} consisting of $45$ lines and $t_{3}=120, t_{4}=45, t_{5}=36$.
\end{enumerate}
At the end of the note, let us present a very recent application of (\ref{LineLanger}).
\begin{remark}
In \cite{Zeye}, the author shows that if $\mathcal{P}=\{P_{1}, ... ,P_{n}\}$ is a finite set of mutually distinct points in the projective plane (not all of them all collinear) and $\mathcal{L}$ is the line arrangement determined by $\mathcal{P}$, then there exists a point $P \in \mathcal{P}$ such that its multiplicity is at least $\frac{n}{3}$. This question was known as a weak Dirac conjecture -- see \cite[Section~6]{Klee} for a nice introduction to the subject.
\end{remark}
\section*{Acknowledgement}
I am very grateful to Adrian Langer for stimulating conversations, useful comments about the content of the note, and for pointing out \cite{Bojanowski}. Finally, I would like to thank the anonymous referees for valuable comments that allowed to improve the note. The author is partially supported by National Science Centre Poland Grant 2014/15/N/ST1/02102.

%*****************************************************************************

%***************************************************************************** % Addresses
\bigskip
   Piotr Pokora,
   Instytut Matematyki,
   Pedagogical University of Cracow,
   Podchor\c a\.zych 2,
   PL-30-084 Krak\'ow, Poland.

Current Address:
    Institut f\"ur Algebraische Geometrie,
    Leibniz Universit\"at Hannover,
    Welfengarten 1,
    D-30167 Hannover, Germany. \\
\nopagebreak
   \textit{E-mail address:} \texttt{piotrpkr@gmail.com, pokora@math.uni-hannover.de}

%*****************************************************************************

\end{document}